\documentclass[a4paper,11pt]{article}

\usepackage[sort&compress,square,comma,numbers]{natbib}

\usepackage{amsmath,amsfonts,amssymb,amsthm}
\usepackage{mathrsfs} 
\usepackage{url}
\usepackage{color}
\usepackage{enumerate}
\usepackage{hyperref,pgf}
\usepackage{setspace}
\usepackage{xspace}

\DeclareMathOperator{\im}{im} 
\DeclareMathOperator{\sign}{sign}
\DeclareMathOperator{\diag}{diag}

\DeclareMathOperator{\conv}{conv}

\newcommand{\mmF}{\mathcal{F}}
\newcommand{\mmG}{\mathcal{G}}
\newcommand{\mmA}{\mathcal{A}}
\newcommand{\mmB}{\mathcal{B}}
\newcommand{\mmq}{q}
\newcommand{\mmm}{m}
\newcommand{\mmM}{\mathcal{M}}
\newcommand{\mmQ}{\mathcal{Q}}
\newcommand{\mmS}{\mathcal{S}}
\newcommand{\mmW}{\mathcal{W}}

\newcommand{\R}{{\mathbb R}}

\newcommand{\dd}[2]{\frac{\text{d} #1}{\text{d} #2}}
\newcommand{\DD}[2]{\frac{\partial #1}{\partial #2}}
\newcommand{\del}{\Delta}

\newcommand{\had}{\ast}
\newcommand{\map}[2]{(#1 \to #2)}

\renewcommand{\k}{\kappa}

\newcommand{\wrt}{on cosets of\xspace}
\newcommand{\inj}{injective\xspace}
\newcommand{\injS}{\inj \wrt~$S$\xspace}
\newcommand{\rel}{affine\xspace}

\theoremstyle{plain}
\newtheorem{theorem}{Theorem}[section]
\newtheorem{proposition}[theorem]{Proposition}

\newtheorem{lemma}[theorem]{Lemma}

\theoremstyle{definition}
\newtheorem{definition}[theorem]{Definition}
\newtheorem{example}[theorem]{Example}
\newtheorem{remark}[theorem]{Remark}
\newtheorem*{notation}{Notation}

\makeatletter
\def\blfootnote{\xdef\@thefnmark{}\@footnotetext}
\makeatother

\begin{document}

\title{Characterizing injectivity of classes of maps \\ via classes of matrices}

\renewcommand{\thefootnote}{\fnsymbol{footnote}}

\author{%
Elisenda Feliu, Stefan M\"uller, Georg Regensburger
}
\blfootnote{
\scriptsize

\noindent
{\bf E.~Feliu} (\href{efeliu@math.ku.dk}{efeliu@math.ku.dk}),
Department of Mathematical Sciences, University of Copenhagen, Universitetsparken 5, 2100 Denmark.

\smallskip
\noindent
{\bf S.~M\"uller} (\href{mailto:st.mueller@univie.ac.at}{st.mueller@univie.ac.at}),
Faculty of Mathematics, University of Vienna, Oskar-Morgenstern-Platz 1, 1090 Wien, Austria

\smallskip
\noindent
{\bf G.~Regensburger} (\href{mailto:georg.regensburger@jku.at}{georg.regensburger@jku.at}),
Institute for Algebra, Johannes Kepler University Linz,
Altenbergerstra{\ss}e 69, 4040 Linz, Austria.

\smallskip
\noindent
All authors contributed equally.
}

\date{\today}

\maketitle

\begin{abstract}
We present a framework for characterizing injectivity of classes of maps (on cosets of a linear subspace) by injectivity of classes of matrices. 
Using our formalism,
we characterize injectivity of several classes of maps, including generalized monomial and monotonic (not necessarily continuous) maps. In fact, monotonic maps are special cases of {\em component-wise affine} maps. Further, we study compositions of maps with a matrix and other composed maps, in particular, rational functions.
Our framework covers classical injectivity criteria based on mean value theorems for vector-valued maps and recent results obtained in the study of chemical reaction networks.


\end{abstract}


\section{Introduction} \label{sec:intro}

If the Jacobian matrix of a differentiable map is injective everywhere,
then the implicit function theorem guarantees local, but not global injectivity of the map.
Only additional conditions on the Jacobian matrix and the domain of the map ensure global injectivity.
For example, the Gale-Nikaid\^o theorem states that, if the Jacobian matrix is a P-matrix everywhere and the domain is rectangular, then the map is injective~\cite{Gale:1965p474}.
By the mean value theorem for vector-valued functions,
the same conclusion holds
if every matrix in the convex hull of Jacobian matrices is injective and the domain is open and convex~\cite{Coomes1989,McLeod1965}.
For further results and references on global injectivity, see~\cite{Parthasarathy}.
In the context of dynamical systems, sufficient conditions for injectivity are used to derive necessary conditions for the existence of multiple equilibria~\cite{gouze,soule}.

In this work, we present a unifying framework that covers classical and recent criteria for injectivity.
Our formalism allows (i) to {\em characterize} injectivity of {\em sets} of maps,
(ii) to consider {\em non-differentiable} maps,
and (iii) to address injectivity {\em on cosets of a linear subspace}. 
The three aspects are motivated by the study of chemical reaction networks,
where one considers dynamical systems of the form
\[
\dd{x}{t} = A \, G(x), \quad x \in X \subseteq \R^n_\ge .
\]
Typically, the map $G$ belongs to a set of maps $\mmG$ (from $X$ to $\R^r_\ge$) such as a class of generalized monomial maps (with fixed exponents, but variable coefficients) or a class of weakly monotonic, not necessarily continuous maps (with fixed sign pattern). Accordingly, the matrix $A$ belongs to $\R^{r \times n}$.
Clearly, $\dd{x}{t}$ is confined to the linear subspace $S = \im(A)$,
and solutions $x(t)$ are confined to affine subspaces. Hence, injectivity of the composed map $A \, G$ on cosets of $S$ precludes the existence of multiple equilibria (solutions~$x$ to $A \, G(x)=0$) on the invariant subspaces. 
Most importantly,
injectivity of particular sets of generalized polynomial maps (with $\im(A)=\R^n$)
has been characterized by
injectivity of the corresponding sets of Jacobian matrices~\cite{ME_I,CGS}.
For subsequent results,
see e.g.~\cite{ME_entrapped,
FeliuWiuf_MAK,gnacadja_linalg} (for injectivity on cosets of $S$),
\cite{BanajiDonnell,ShinarFeinberg2012,WiufFeliu_powerlaw} (for other sets of maps),
\cite{TSS,MR2012,MFR2016} (for sign conditions and applications to real algebraic geometry),
and \cite{BanajiPantea2016} (for a comprehensive treatment).

In our framework, we single out a key argument underlying previous works.
Given a set of maps $\mmG$ (from $X \subseteq \R^n$ to $\R^r$), assume there exists a set of matrices $\mmB$ (in $\R^{r \times n}$) and a subset $T$ (of $\R^n$) such that 
\[ 
\{ G(x)-G(y) \mid G \in \mmG, \, x,y \in X, \, x \neq y \} = \{ B(z) \mid B \in \mmB, \, z \in T\}. 
\]
Then all maps in $\mmG$ are \inj if and only if the set on the right-hand side does not contain zero, that is, no element in $T$ belongs to the kernel of any matrix in $\mmB$. In this way, injectivity of the set of maps $\mmG$ is reduced to a linear algebra problem.

We construct suitable sets of matrices $\mmB$ for several sets of maps.
For example, for a class of monotonic maps,
$\mmB$ is a qualitative class of matrices given by a sign pattern.
In fact, monotonic maps are special cases of {\em component-wise affine} maps for which $\mmB$ is given as the Cartesian product of subsets of $\R$.
Further, we consider composed maps such as rational functions (compositions of monomial and monotonic maps),
we provide examples to compare different choices of $\mmB$ for the same set of maps, and we study the effect of the domain $X$ 
(being rectangular, the positive orthant, open and convex, etc.). 
For classes of generalized monomial maps,
our results do not follow from classical injectivity criteria based on mean value theorems for vector-valued functions or the fundamental theorem of calculus.
On the other hand, the Gale-Nikaid\^o theorem (for individual maps) does not follow from our formalism (for sets of maps).

The paper is organized as follows. In Section~\ref{sec:frame}, we present our general framework. In Section~\ref{sec:classes}, we characterize injectivity of classes of maps,
in particular, component-wise affine, generalized monomial, and monotonic maps, and we study compositions with a matrix.
In Section~\ref{sec:diff}, we discuss how our framework covers classical results using mean value theorems for vector-valued functions. Finally, in Section~\ref{sec:compose}, we consider more examples of composed maps, in particular, rational functions.


\clearpage

\section{Mathematical framework} \label{sec:frame}

We characterize injectivity of a set of maps $\mmG$ (on cosets of a linear subspace) using a set of matrices $\mmB$. 
In order to find a suitable set of matrices,
we introduce the notion of \emph{affinity}.
Further, we consider composition of maps.


\subsection{Notation}
 
We denote the strictly positive real numbers by $\R_>$ and the non-negative real numbers by $\R_\ge$.
For $x,y \in \R^n$,
we denote the component-wise (or Hadamard) product by $x \had y \in \R^n$, that is, $(x \had y)_i = x_i y_i$.
For a subset $X \subseteq \R^n$,
we write $X^* = X \setminus \{0\}$
and
\[
dX = \{ x-y \mid x,y \in X, \, x\neq y \}.
\]
For any natural number $n$, we define $[n]=\{1,\dots,n\}$. 

\paragraph{Maps and matrices.}
Let $X$, $Y$ be sets.
For the set of all maps from $X$ to $Y$ we write $\map{X}{Y}$.
Let $\mmG \subseteq \map{X}{Y}$ be a set of maps.
We say that $\mmG$ has a certain property,
if every map $G \in \mmG$ has this property.
If $\mmG = \{G\}$, we simply write $G$ for $\mmG$.
We denote the image of $S \subseteq X$ under $G \colon X\to Y$ by $G(S)$
and define
\begin{equation*}
\mmG(S) = \bigcup_{G \in \mmG} G(S).
\end{equation*}

Let $X$, $Y$, and $Z$ be sets
and $\mmG \subseteq \map{X}{Y}$ and $\mmF \subseteq \map{Y}{Z}$ be sets of maps.
We write
\begin{equation*}
\mmF \circ \mmG = \{ F \circ G \mid G \in \mmG, \, F \in \mmF \} \subseteq \map{X}{Z}
\end{equation*}
for the set of composed maps.
 
We identify a matrix $B \in \R^{r \times n}$ with the 
corresponding linear map $B \colon \R^n \to \R^r$
and write $\im(B)$ and $\ker(B)$ for the respective linear subspaces.
For a set of matrices $\mmB \subseteq \R^{r \times n}$, we define
\[
\ker(\mmB) = \bigcup_{B \in \mmB} \ker(B).
\]
For sets of matrices  $\mmB \subseteq \R^{r \times n}$ and $\mmA \subseteq \R^{m \times r}$, we write the set of product matrices as
\[
\mmA\mmB = \{AB\in \R^{r\times n} \mid B \in \mmB, \, A\in \mmA\}.
\]

\paragraph{Sign vectors.}
For a vector $x \in \R^n$, we obtain the \emph{sign vector} $\sigma(x)\in \{-,0,+\}^n$
by applying the sign function component-wise.
For a subset $S \subseteq \R^n$, we write
\begin{equation*} 
\sigma(S) = \{ \sigma(x) \mid x \in S \}  
\end{equation*}
for the set of all sign vectors of $S$
and 
\[
\Sigma(S) = \sigma^{-1}(\sigma(S))
\]
for the union of all (possibly lower-dimensional) orthants that $S$ intersects.
For $x,y \in \R^n$, we have  $\sigma(x) = \sigma(y)$ if and only if $x = \lambda \had y$ for some $\lambda \in \R_>^n$,
and hence
\[
\Sigma(S) = \{\lambda \had x \mid \lambda \in \R_>^n \text{ and } x \in S\}.
\]
For subsets $X,Y \subseteq \R^n$, we have the equivalences
\begin{equation} \label{eq:sigmaSigma}
\Sigma(X) \cap Y = \emptyset
\quad \Leftrightarrow \quad
\sigma(X) \cap \sigma(Y) = \emptyset
\quad \Leftrightarrow \quad
X \cap \Sigma(Y) = \emptyset .
\end{equation}
The inequalities $0<-$ and $0<+$ induce a partial order on $\{-,0,+\}^n$:
for sign vectors $\tau, \rho \in \{-,0,+\}^n$, we write $\tau \le \rho$ if the inequality holds component-wise.


\subsection{Injectivity}

For $X \subseteq \R^n$ and $G \colon X \to \R^r$,
let
\[
\del G = \{ G(x)-G(y) \mid x,y \in X, \, x \neq y \} .
\]
By definition, $G$ is \inj if $G(x) \neq G(y)$ for all $x,y \in X$ with $x \neq y$,
that is, if $0 \notin \del G$. Similarly, let
\[
\del \mmG = \bigcup_{G\in \mmG} \del G,
\]
and $\mmG$ is injective if and only if $0\notin \del \mmG$.

Motivated by applications to dynamical systems, in particular,
by the study of solutions to $\dd{x}{t} = A\, F(x)=0$ on cosets of $\im(A)$, 
we study injectivity on cosets of an arbitrary linear subspace $S \subseteq \R^n$.
For $x \in X$, we write
\[
X_x = (x+S) \cap X
\]
for the intersection of the coset given by $x$ and the domain $X$
and further $G_x$ for the map $G$ with domain restricted to $X_x$.
We observe that 
$\bigcup_{x \in X} X_x = X$
and
\begin{equation} \label{eq:dXx}
\bigcup_{x \in X} dX_x = dX \cap S.
\end{equation}
We say that $G$ is \injS if $G_x$ is \inj for all $x \in X$.
We define
\begin{align*}
\del_S G &= \bigcup_{x \in X} \del G_x \\
&= \big\{ G(x)-G(y) \mid x,y \in X, \, x-y \in S^* \big\} 
\end{align*}
and, for a set of maps $\mmG \subseteq \map{X}{\R^r}$,
\[\del_S \mmG = \bigcup_{G \in \mmG} \del_S G. \]
Hence, $\mmG$ is \injS if and only if $0 \notin \del_S \mmG$.
Unrestricted injectivity corresponds to $S = \R^n$ and hence $\del_S \mmG = \del \mmG$. 

If the set of differences $\del_S \mmG$ can be written as the image of some set $T$ under a set of linear maps $\mmB$,
we can characterize injectivity of $\mmG$ by the following observation.

\begin{lemma} \label{lem:frame}
Let $X,S,T \subseteq \R^n$ with $S$ a linear subspace, $\mmG \subseteq \map{X}{\R^r}$, and $\mmB \subseteq \R^{r\times n}$. 
If $\del_S \mmG = \mmB(T)$,
then the following statements are equivalent:
\begin{enumerate}[(i)]
\item $\mmG$ is \injS.
\item $0 \notin \mmB(T)$. 
\end{enumerate}
If $\del_S \mmG \subseteq \mmB(T)$, then (ii) implies (i).
\end{lemma}

\begin{proof} 
By definition, 
(i) is equivalent to $0 \notin \del_S \mmG$ and hence to (ii). 
\end{proof}

In general, the characterization of injectivity in Lemma~\ref{lem:frame} is useful only for sets of maps $\mmG$, but not for individual maps $G$.


\subsection{Affinity} \label{sec:aff}

For proving injectivity of a set of maps $\mmG$
by Lemma~\ref{lem:frame},
one needs to find a set of matrices $\mmB$ (and a set $T$) such that $\del_S \mmG \subseteq \mmB(T)$.
Often, $\mmB$ is obtained from a more restrictive requirement, namely that $\mmG$ is $\mmB$-\rel. 
\begin{definition} 
Let $X \subseteq \R^n$, $G \colon X \to \R^r$, and $\mmB \subseteq \R^{r\times n}$ be a set of matrices.
Then $G$ is called \emph{$\mmB$-\rel} if, for all $x,y \in X$,
there exists $B\in \mmB$ such that
\[
G(x)-G(y)=B(x-y).
\]
\end{definition}

For characterizing injectivity, one additionally requires $\mmB \subseteq \mmG$, that is, the set of maps contains all linear maps guaranteeing its affinity.

\begin{lemma} \label{lem:aff} 
Let $X,S \subseteq \R^n$ with $S$ a linear subspace, $\mmG \subseteq \map{X}{\R^r}$, and $\mmB \subseteq \R^{r\times n}$.
If $\mmG$ is $\mmB$-\rel, then
\[
\del_S \mmG \subseteq \mmB(dX \cap S) .
\]
If additionally $\mmB\subseteq \mmG$, then
\[
\del_S \mmG = \mmB(dX \cap S).
\]
\end{lemma}
\begin{proof}
For $x \in X$, we write 
$\mmG_x$ for the set of maps $\mmG$ with domain restricted to the coset $X_x$.
Assume $\mmG$ is $\mmB$-\rel. Then, $\mmG_x$ is $\mmB$-\rel for all $x \in X$
and by definition we have $\del \mmG_x \subseteq \mmB(dX_x)$. 
Hence, using \eqref{eq:dXx},
\[
\del_S \mmG = \bigcup_{x \in X} \del \mmG_x \subseteq \bigcup_{x \in X} \mmB(dX_x) = 
\mmB\Big(\bigcup_{x \in X} dX_x\Big) = \mmB(dX \cap S) .
\]
Additionally assume $\mmB \subseteq \mmG$. Let $B \in \mmB$ and $z \in dX\cap S$. Then, $z = x-y\in S^*$ with $x,y \in X$ and $x\neq y$, 
and $B z = B (x-y) = B(x)- B(y) \in \del_S \mmG$ since $B \in \mmB \subseteq \mmG$.
\end{proof}

In the case of unrestricted injectivity, $S = \R^n$,
if $\mmG$ is $\mmB$-\rel, then $\del \mmG \subseteq \mmB(dX)$.
If additionally $\mmB\subseteq \mmG$, then $\del \mmG = \mmB(dX)$.

Now, we can combine Lemmas~\ref{lem:frame} and~\ref{lem:aff}.

\begin{proposition} \label{pro:main} \label{pro:maincoset}
Let $X,S  \subseteq \R^n$ with $S\subseteq \R^n$ a linear subspace, $\mmG \subseteq \map{X}{\R^r}$, and $\mmB \subseteq \R^{r\times n}$.
If $\del_S \mmG = \mmB(dX\cap S)$,
in particular, if $\mmG$ is $\mmB$-\rel and $\mmB \subseteq \mmG$,
then the following statements are equivalent:
\begin{enumerate}[(i)]
\item $\mmG$ is \injS.
\item $0 \notin \mmB (dX \cap S)$.
\end{enumerate}
If $\del_S \mmG \subseteq \mmB(dX \cap S)$,
in particular, if $\mmG$ is $\mmB$-\rel,
then (ii) implies (i).
\end{proposition}

If $X \subseteq \R^n$ has non-empty interior,
then $dX \cup \{0\}$ contains a ball around zero,
and condition (ii) in Proposition~\ref{pro:maincoset} can be rewritten:
\begin{equation} \label{eq:ker_B_S}
0 \notin \mmB (dX \cap S)
\quad \Leftrightarrow \quad
0 \notin \mmB (S^*) 
\quad \Leftrightarrow \quad
\ker(\mmB)\cap S =\{0\} .
\end{equation}
Let $Z \in \R^{(n-s)\times n}$ with $s=\dim S$ such that $S=\ker Z$
and let
\[
M_B =
\begin{pmatrix}
Z \\ B
\end{pmatrix} .
\]
Then further
\[
\ker(\mmB)\cap S =\{0\} 
\quad \Leftrightarrow \quad
\ker(M_B) = \{0\} \quad \text{for all } B \in \mmB .
\]
Finally, if $r=s$, then $M_B$ is a square matrix and the condition is equivalent to
\begin{equation} \label{total_determinental}
\det(M_B)\neq 0 \quad \text{for all } B\in \mmB .
\end{equation}
Determinant conditions for injectivity play a prominent role in the study of chemical reaction networks, 
see for example \cite{ME_I,CGS,
BanajiDonnell,
FeliuWiuf_MAK,WiufFeliu_powerlaw,feliu-bioinfo,feliu_newinj,
MFR2016,BanajiPantea2016}.


\subsection{Composition} 

Assuming that a set of maps $\mmF$ is $\mmA$-\rel, we consider composed maps $\mmF \circ \mmG$ and the resulting sets $\del_S(\mmF \circ \mmG)$.

\begin{lemma} \label{lem:comp}
Let $X,S,T \subseteq \R^n$ with $S$ a linear subspace, $Y \subseteq \R^r$, 
$\mmG \subseteq \map{X}{Y}$, $\mmB \subseteq \R^{r \times n}$, 
$\mmF \subseteq \map{Y}{\R^m}$,  and $\mmA \subseteq \R^{m \times r}$.
In particular, let $\mmF$ be $\mmA$-\rel.
\begin{enumerate}[(a)]
\item  If $\del_S \mmG \subseteq \mmB(T)$,
then
\[
\del_S (\mmF \circ \mmG) \subseteq \mmA\mmB (T) .
\]
\item If $\del_S \mmG = \mmB(T)$ and $\mmA \subseteq \mmF$, 
then
\[
\del_S (\mmF \circ \mmG) = \mmA\mmB (T) .
\]
\item If $\mmG$ is $\mmB$-\rel, then $\mmF\circ \mmG$ is $\mmA\mmB$-\rel.
\end{enumerate} 
\end{lemma}
\begin{proof}
(a) Let $F \in \mmF$, $G \in \mmG$, and $x,y \in X$ such that $x-y\in S^*$.
Then, there exist $A \in \mmA$, $B \in \mmB$, and $z \in T$ such that $F(G(x))-F(G(y)) = A(G(x)-G(y)) = AB z$.

(b) It remains to show inclusion ($\supseteq$).
Let $A \in \mmA$, $B \in \mmB$, and $z \in T$.
Then there exist $G \in \mmG$ and $x,y \in X$ such that $x-y\in S^*$, and $AB z = A(G(x)-G(y)) = A(G(x))-A(G(y))\in \del_S (\mmF \circ \mmG)$ since $\mmA\subseteq \mmF$.

(c) Analogous to (a).
\end{proof}

Motivated by applications,
we extend Proposition~\ref{pro:maincoset},
by using Lemma~\ref{lem:comp}, the trivial fact that $A \in \R^{m \times r}$ is $\{A\}$-\rel, 
and the above equivalence.

\begin{proposition} \label{pro:framecomp}
Let $X,S \subseteq \R^n$ with $S$ a linear subspace, $Y \subseteq \R^r$, $\mmG \subseteq \map{X}{Y}$, 
$\mmB \subseteq \R^{r \times n}$, and $A \in \R^{m \times r}$.
If $\del_S \mmG = \mmB(dX \cap S)$, 
in particular, if $\mmG$ is $\mmB$-\rel and $\mmB \subseteq \mmG$,
then $\del_S (A\circ \mmG)= A\mmB(dX\cap S)$, and the following statements are equivalent:
\begin{enumerate}[(i)]
\item $A \circ \mmG$ is \injS.
\item $0 \notin A\mmB(dX\cap S)$.
\end{enumerate}
If $\del_S \mmG \subseteq \mmB(dX \cap S)$,  then $\del_S (A\circ \mmG)\subseteq A\mmB(dX\cap S)$, and (ii) implies (i).
\end{proposition}
 
If $X \subseteq \R^n$ has non-empty interior, then condition (ii) in Proposition~\ref{pro:framecomp} can be rewritten, cf.~equivalence~\eqref{eq:ker_B_S}:
\begin{equation} \label{eq:ker_AB_S}
\begin{aligned}
0 \notin A\mmB(dX\cap S)
\quad & \Leftrightarrow \quad 
0 \notin A\mmB(S^*) 
\quad \Leftrightarrow \quad
\ker(A\mmB) \cap S = \{0\}
\\
\quad & \Leftrightarrow \quad  
\ker(A) \cap \mmB(S^*) = \emptyset .
\end{aligned}
\end{equation}


\section{Classes of maps} \label{sec:classes}

For characterizing injectivity of a set of maps $\mmG \subseteq \map{X}{\R^r}$ (on cosets of a linear subspace~$S$) by Proposition~\ref{pro:maincoset},
one needs to find a set of matrices $\mmB$ such that $\del_S \mmG = \mmB(dX \cap S)$.
Often, $\mmB$ is obtained from a more restrictive requirement,
namely that $\mmG$ is $\mmB$-\rel (and $\mmB \subseteq \mmG$),
in particular, that $\mmG$ is {\em component-wise} \rel, see Subsection~\ref{sec:component-wise}.

For the classes of generalized monomial and monotonic maps (and compositions thereof),
we characterize injectivity,
see Subsections~\ref{sec:monomial} and \ref{sec:monoton}.
Importantly, monotonic maps are special cases of component-wise affine maps.
We highlight the relations between monomial and polynomial maps and between monotonic and monomial/polynomial maps,
and we discuss how previous results (from the study of chemical reaction networks) are covered by our framework.


\subsection{Component-wise affine maps} \label{sec:component-wise}

We consider sets of maps $\mmG$ (on a suitable domain) for which the sets of matrices $\mmB$ (such that $\mmG$ is $\mmB$-\rel) are given as the Cartesian product of subsets of $\R$.

\begin{definition}
Let $D=(d_{ij})$ be an $r \times n$ matrix of non-empty subsets $d_{ij} \subseteq \R$.
The set of matrices,
\begin{equation*}
\mmQ(D) = \{(b_{ij}) \in \R^{r \times n} \mid b_{ij} \in d_{ij} \} ,
\end{equation*}
is called the \emph{qualitative class} of $D$.
Further, let $X \subseteq \R^n$ and $G \colon X \to \R^r$.
We say that $G$ is \emph{component-wise $D$-\rel}, if 
\[
G_i(x)-G_i(y) = b_{ij} (x_j-y_j) \quad \text{with } b_{ij} \in d_{ij}
\]
for all $x,y \in X$ that differ only in the $j$-th component,
 that is, $x-y = (0, \ldots, x_j-y_j, \ldots, 0)^T$ (for $i \in [r]$, $j \in [n]$). 
\end{definition}

Clearly, every matrix in $\mmQ(D)$ is component-wise $D$-\rel. 
More generally, every $\mmQ(D)$-affine map is component-wise $D$-\rel.

\begin{lemma} \label{lem:D} 
Let $X \subseteq \R^n$, $G \colon X \to \R^r$,
and $D$ be an $r \times n$ matrix of non-empty subsets of $\R$.
If $G$ is $\mmQ(D)$-\rel, then $G$ is component-wise $D$-\rel.
\end{lemma}

\begin{proof}
Let $i \in [r]$, $j \in [n]$ and $x,y \in X$ that differ only in the $j$-th component.
If $G$ is $\mmQ(D)$-\rel, then there exists $B=(b_{ij})\in \mmQ(D)$ such that $G(x)-G(y)=B(x-y)$.
In particular
\[
G_i(x)-G_i(y) = b_{ij} (x_j - y_j)
\]
with $b_{ij}\in d_{ij}$. 
\end{proof}

As it turns out (in the following lemmas),
the two notions of affinity are equivalent
under additional assumptions on the domain $X$ and the matrix $D$.

By a \emph{rectangular domain}, we mean a Cartesian product of real intervals, without imposing any restriction on whether the intervals are open or closed.

\begin{lemma} \label{lem:Drectangular} 
Let $X \subseteq \R^n$ be a rectangular domain, $G \colon X \to \R^r$,
and $D$ be an $r \times n$ matrix of non-empty subsets of $\R$.
The following statements are equivalent:
\begin{enumerate}[(i)]
\item
$G$ is $\mmQ(D)$-\rel.
\item
$G$ is component-wise $D$-\rel.
\end{enumerate}
\end{lemma}

\begin{proof}
(i) $\Rightarrow$ (ii): By Lemma~\ref{lem:D}.

(ii) $\Rightarrow$ (i):
Let $x,y \in X$
and define a sequence $(z^j)_{j=0}^n$ from $y$ to $x$ by
\begin{equation*} 
z^0=y, \quad z^j = (x_1,\dots,x_j,y_{j+1},\dots,y_n) \text{ for } j \in [n-1], \quad \textrm{and } z^n=x.
\end{equation*}
Clearly, $z^j - z^{j-1} = (0, \ldots, x_j-y_j, \ldots, 0)^T$ for $j \in [n]$.
Now, assume that $G$ is component-wise $D$-\rel.
For $i \in [r]$, we obtain
\[
G_i(x) - G_i(y) = \sum_{j=1}^n G_i(z^j) - G_i(z^{j-1}) = \sum_{j=1}^n b_{ij} (x_j-y_j)
\]
with $b_{ij}\in d_{ij}$ for $j \in [n]$.
Note that, if $x_j=y_j$, then $b_{ij}$ can be chosen arbitrarily in $d_{ij}$.
Hence, $G(x)-G(y)=B(x-y)$ with $B=(b_{ij}) \in \mmQ(D)$.
\end{proof}

An analogous result holds for a more general domain (not necessarily rectangular), at the cost of restricting $D$ to a matrix of intervals (not arbitrary subsets).

Let $X \in \R^n$ and $x,y \in X$. 
A \emph{path from $y$ to $x$} is a sequence $(z^k)_{k=0}^N$, $z^k\in X$, with $z^0=y$ and $z^N=x$.
A path is called \emph{rectangular},
if successive elements differ only in one component.
That is, for all steps $k \in [N]$, there exists $j \in [n]$
such that $(z^k - z^{k-1})_{j} \neq 0$ and $(z^k - z^{k-1})_{j'} = 0$ for $j' \neq j$.
A path is called \emph{oriented},
if differences of successive elements conform to the overall difference $x-y$,
that is, $\sigma(z^k - z^{k-1}) \le \sigma(x-y)$ for all $k \in [N]$.
A set $X$ is called \emph{connected by rectangular, oriented paths}
if for all $x,y \in X$ there exists a rectangular, oriented path from $y$ to $x$. 
For example, an open convex set is connected by rectangular, oriented paths.

\begin{lemma} \label{lem:Dconnected} 
Let $X\subseteq \R^n $ be connected by rectangular, oriented paths, 
$G\colon X\to \R^r$, and $D$ an $r \times n$ matrix of intervals.
The following statements are equivalent:
\begin{enumerate}[(i)]
\item
$G$ is $\mmQ(D)$-\rel.
\item
$G$ is component-wise $D$-\rel.
\end{enumerate}
\end{lemma}

\begin{proof}
(i) $\Rightarrow$ (ii): By Lemma~\ref{lem:D}.

(ii) $\Rightarrow$ (i):
Let $x,y \in X$.
By assumption,
there exists a rectangular, oriented path $(z^k)_{k=0}^N$ from $y$ to $x$.
We group the steps $[N]$ into equivalence classes $C_1, \ldots, C_n$:
\[
C_j = \{ k \in [N] \mid (z^k - z^{k-1})_{j} \neq 0 \} \quad \text{for } j \in [n],
\]
in particular, $C_j = \emptyset$ if $x_j = y_j$, because the path is oriented.
Clearly, 
\[
\sum_{k \in C_j} (z^k - z^{k-1})_{j} = x_j - y_j.
\]
Now, assume that $G$ is component-wise $D$-\rel. For $x_j \neq y_j$, we introduce
\[
\xi^j_k = (z^k_j - z^{k-1}_j) / (x_j - y_j) \ge 0 \quad \text{for } k \in C_j 
\]
such that $\sum_{k \in C_j} \xi^j_k = 1$.
For $i \in [r]$, we obtain
\begin{align*}
G_i(x) - G_i(y) &= \sum_{k \in [N]} G_i(z^k) - G_i(z^{k-1}) \\
&= \sum_{j \in [n]} \sum_{k \in C_j} G_i(z^k) - G_i(z^{k-1}) \\
&= \sum_{j \in [n]} \sum_{k \in C_j} b_{ij}^k (z^k_j - z^{k-1}_j) \quad \text{with } b_{ij}^k \in d_{ij} \text{ for } k \in C_j \\
&= \sum_{j \in [n]} \sum_{k \in C_j} b_{ij}^k \, \xi^j_k (x_j - y_j) \\
&= \sum_{j \in [n]} A_{ij} (x_j - y_j) \quad \text{with } A_{ij} \in d_{ij},
\end{align*}
where $A_{ij} = \sum_{k \in C_j} b_{ij}^k \, \xi^j_k \in d_{ij}$ since the interval $d_{ij}$ is convex.
Note that, if $x_j=y_j$, then $A_{ij}$ can be chosen arbitrarily in $d_{ij}$.
Hence, $G(x)-G(y)=A(x-y)$ with $A=(A_{ij}) \in \mmQ(D)$.
\end{proof}

For fixed $X \subseteq \R^n$, we denote the set of component-wise $D$-\rel maps by
\[
\mmM(D) \subseteq \map{X}{\R^r} .
\]
Under the assumptions on the domain $X$ and the matrix $D$ of Lemmas~\ref{lem:Drectangular} or \ref{lem:Dconnected},
$\mmM(D)$ is $\mmQ(D)$-\rel. 
Further, 
$\mmQ(D) \subseteq \mmM(D)$, where elements of $\mmQ(D)$ have their domains restricted from $\R^n$ to $X$.
Hence,
we can apply Proposition~\ref{pro:maincoset}.

\begin{theorem} \label{thm:D}
Let $X,S \subseteq \R^n$ with $S$ a linear subspace, $D$ be an $r \times n$ matrix of non-empty subsets of~$\R$,
and $\mmM(D) \subseteq \map{X}{\R^r}$ be the set of component-wise $D$-\rel maps.
If $X$ is rectangular or $X$ is connected by rectangular, oriented paths and $D$ is a matrix of intervals, 
then
\begin{equation*} 
\text{$\mmM(D)$ is $\mmQ(D)$-\rel}, \, 
\mmQ(D) \subseteq \mmM(D), \,
\del_S \mmM(D) = \mmQ(D)(dX \cap S) ,
\end{equation*}
and the following statements are equivalent:
\begin{enumerate}[(i)]
\item $\mmM(D)$ is \injS.
\item $0 \notin \mmQ(D)(dX \cap S)$.
\end{enumerate}
\end{theorem}


\subsection{Generalized monomial maps on the positive orthant} \label{sec:monomial}

Let $X=\R^n_>$ (and hence $dX=\R^n\setminus \{0\}$).
For $B = (b_1,\ldots,b_r)^T \in \R^{r \times n}$,
we consider the \emph{generalized monomial map} $\mu_B \colon \R_>^n \to \R_>^r, \, x \mapsto x^B$, given by
\[
(x^B)_j = x^{b_{j}}
= x_1^{b_{j1}} \cdots x_n^{b_{jn}} \quad \text{for } j \in [r] ,
\]
and the sets
\begin{align*}
\mmq(B) &=\{ \diag(\k)B\diag(\lambda) \mid \k\in \R_>^r, \, \lambda\in \R_>^n \}\subseteq \R^{r\times n}, \\
\mmm(B) &= \{ x \mapsto \k \had x^B \mid \k \in \R_>^r \} \subseteq \map{\R^n_>}{\R^r_>}.
\end{align*}

Note that the individual map $\mu_B$ is \inj if and only if the set of maps $\mmm(B)$ is \inj.
\begin{proposition} \label{pro:monomial}
Let $B \in \R^{r \times n}$ and $S\subseteq \R^n$ be a linear subspace. Then,
\begin{enumerate}[(a)]
\item $\mu_B$ is $\mmq(B)$-\rel,
\item $\sigma(\del_S \, \mu_B) = \sigma(B(\Sigma(S^*)))$,
\item $\mmm(B)$ is $\mmq(B)$-\rel ,
\item $\del_S \, \mmm(B) = \mmq(B)(S^*) $.
\end{enumerate}
\end{proposition}
\begin{proof}
(a) For $x\in \R^n_>$, let $\ln x=(\ln x_1,\dots,\ln x_n)^T$. Since the logarithm is a strictly increasing function, for every $x,y\in \R^n_>$, there exist $\k\in \R^r_>$ and $\lambda\in \R^n_>$ such that
\[
x^B - y^B = \k \had (\ln x^B - \ln y^B) = \k \had (B(\ln x - \ln y)) = \k \had B( \lambda \had (x-y)).
\]
(b) By Lemma 2.4 in~\cite{MFR2016}. \\
(c) By (a). \\
(d) Using (b), $\del_S \, \mmm(B) = \Sigma(\del_S \, \mu_B) = \Sigma(B(\Sigma(S^*))) = \mmq(B)(S^*)$.
\end{proof}
By (a), $\mu_B$ is $\mmq(B)$-\rel, but $\del_S \, \mu_B \neq B(\Sigma(S^*))$ and $\mmq(B) \not\subseteq \{ \mu_B \}$.
Hence, Lemma~\ref{lem:frame} and Proposition~\ref{pro:main} do not apply.
By (c), also $\mmm(B)$ is $\mmq(B)$-\rel, but $\mmq(B) \not\subseteq \mmm(B)$. Still, (d) holds, and Lemma~\ref{lem:frame} applies. We obtain:
\[
\mmm(B) \text{ is \injS }
\quad \Leftrightarrow \quad
0 \notin \mmq(B)(S^*) .
\]
By using $\mmq(B)(S^*) = \Sigma(B(\Sigma(S^*)))$ and equivalence~\eqref{eq:sigmaSigma}, we further obtain:
\[
0 \notin \mmq(B)(S^*)
\quad \Leftrightarrow \quad
0 \notin B(\Sigma(S^*))
\quad \Leftrightarrow \quad
\sigma(\ker(B)) \cap \sigma(S^*) = \emptyset .
\]
Hence, injectivity of the individual map $\mu_B$ and the set of maps $\mmm(B)$ is characterized by the sign condition obtained in \cite[Proposition~2.5]{MFR2016}.

\subsubsection{Monomial and polynomial maps} \label{sec:polynom}

Motivated by applications, we additionally consider a matrix $A \in \R^{m \times r}$ and the resulting set $A \circ \mmm(B)$ of generalized polynomial maps on the positive orthant.
By Proposition~\ref{pro:framecomp} and equivalence~\eqref{eq:ker_AB_S}, we obtain
\begin{equation} \label{eq:mon_poly}
A \circ m(B) \text{ is \injS } 
\quad \Leftrightarrow \quad
\ker(A) \cap \mmq(B)(S^*) = \emptyset .
\end{equation}
By using $\mmq(B)(S^*) = \Sigma(B(\Sigma(S^*)))$ and equivalence~\eqref{eq:sigmaSigma},
the latter condition is equivalent to the sign condition
\[
\sigma(\ker(A)) \cap \sigma(B(\Sigma(S^*))) = \emptyset ,
\]
obtained in \cite[Theorem~1.4]{MFR2016}.
In the case of unrestricted injectivity, $S=\R^n$,
it is further equivalent to the sign condition
\[
\sigma(\ker(A)) \cap \sigma(\im (B)) = \{ 0 \} ,
\]
obtained in \cite[Theorem~3.6]{MR2012}.


\subsection{Monotonic maps} \label{sec:monoton}

The class of generalized monomial maps, studied in the previous section, is included in the corresponding class of monotonic maps.
For a matrix $B \in \R^{r \times n}$,
we obtain the {\em sign pattern matrix} $\sigma(B) \in \{-, 0,+\}^{r \times n}$ by applying the sign function entry-wise.
Conversely, for $W \in \{-,0,+\}^{r \times n}$, we introduce the {\em qualitative class}
\[
\mmQ(W) = \{ B \in \R^{r \times n} \mid \sigma(B)=W \}
\]
of matrices with sign pattern $W$.
In order to define (non-strict) monotonicity, we introduce the set of all possible sign combinations
\begin{equation*} 
\mmS = \big\{ \{0\}, \{-\}, \{+\}, \{-,0\}, \{0,+\}, \{-,+\}, \{-,0,+\} \big\} .
\end{equation*}
For notational simplicity, we identify a matrix $\mmW \in \mmS^{r \times n}$ 
with the set of sign patterns
\[
\{ W \in \{-,0,+\}^{r \times n} \mid W_{ij} \in \mmW_{ij} \} .
\]
For $\mmW \in \mmS^{r \times n}$, we introduce the {\em qualitative class}
\[
\mmQ(\mmW) = \{ B \in \R^{r \times n} \mid \sigma(B) \in \mmW \}
\]
and note that 
\begin{equation} \label{eq:qualitative_class_union}
\mmQ(\mmW)= \bigcup_{W \in \mmW} \mmQ(W).
\end{equation}

\begin{definition}
Let $X \subseteq \R^n$, $G \colon X \to \R^r$, and $\mmW \subseteq \mmS^{r \times n}$. 
The map $G$ is called \emph{$\mmW$-monotonic}, if
\[
\sigma(G_i(x)-G_i(y)) \in \mmW_{ij}
\]
for all $x,y\in X$ that differ only in the $j$-th component
(for $i \in [r]$ and $j \in [n]$).
\end{definition}
If $G$ is $\mmW$-monotonic, then 
$G_i$ is strictly decreasing in $x_j$ if $\mmW_{ij}=\{-\}$,
decreasing if $\mmW_{ij}=\{-,0\}$,
constant if $\mmW_{ij}=\{0\}$,
etc.

For fixed $X \subseteq \R^n$, we denote the set of all $\mmW$-monotonic maps by
\[
\mmM(\mmW) \colon X \to \R^r .
\]
 Now, note that $\mmQ(\mmW)=\mmQ(D_\mmW)$
with the $r \times n$ matrix $D_\mmW$ given by 
\[
(D_\mmW)_{ij} =
\begin{cases} 
0 & \textrm{if } \mmW_{ij}=\{0\}, \\ 
(-\infty,0) & \textrm{if }  \mmW_{ij}=\{-\}, \\
(0,+\infty)  & \textrm{if } \mmW_{ij}=\{+\}, \\
(-\infty,0] & \textrm{if } \mmW_{ij}=\{-,0\}, \\
[0,+\infty) & \textrm{if } \mmW_{ij}=\{0,+\}, \\ 
(-\infty,0)\cup (0,+\infty) & \textrm{if } \mmW_{ij}=\{-,+\}, \\ 
(-\infty,+\infty) & \textrm{if }  \mmW_{ij}=\{-,0,+\} .
\end{cases} 
\] 
In fact,
\[
\mmM(\mmW) = \mmM(D_\mmW),
\]
that is, the set of $\mmW$-monotonic maps agrees with the set of component-wise $D_\mmW$-\rel maps,
 and Theorem~\ref{thm:D} applies (for a rectangular domain~$X$ and a linear subspace $S$). We obtain:
\[
\text{$\mmM(\mmW)$ is $\mmQ(\mmW)$-\rel}, \, 
\mmQ(\mmW) \subseteq \mmM(\mmW), \,
\del_S \mmM(\mmW) = \mmQ(\mmW)(dX \cap S) ,
\]
and
\begin{equation}
\label{eq:MW}
\mmM(\mmW) \text{ is \injS }
\quad \Leftrightarrow \quad
0 \notin \mmQ(\mmW) (dX \cap S) .
\end{equation}
By Equation~\eqref{eq:qualitative_class_union}, 
$\mmM(\mmW)$ is \inj (on cosets of a linear subspace $S$) if and only if $\mmM(W)$ is \injS for all $W \in \mmW$.

\subsubsection{Monotonic and monomial/polynomial maps} \label{sec:monoton_monom}

Injectivity of sets of monotonic maps is closely related to injectivity of sets of generalized monomial maps. Let $X=\R^n_>$ and $\mmW \in \mmS^{r \times n}$.
We introduce the set of $\mmW$-monomial maps,
\[
\mmm(\mmW) = \bigcup_{B \in \mmQ(\mmW)} m(B).
\]
Clearly, $\mmm(\mmW) \subset \mmM(\mmW)$. 

\begin{proposition} \label{pro:w_monomial}
Let $\mmW\in \mmS^{r \times n}$ and $S\subseteq \R^n$ a linear subspace. Then,
\begin{enumerate}[(a)]
\item $\mmm(\mmW)$ is $\mmQ(\mmW)$-\rel,
\item $\del_S \, \mmm(\mmW) = \mmQ(\mmW)(S^*)$.
\end{enumerate}
\end{proposition}
\begin{proof}
Given $B\in \mmQ(\mmW)$, note that $\{B\} \subseteq \mmq(B) \subseteq \mmQ(\mmW)$ and hence 
\[
\bigcup_{B \in \mmQ(\mmW)} \mmq(B) = \mmQ(\mmW).
\]
(a) By Proposition~\ref{pro:monomial}(c), $m(B)$ is $q(B)$-\rel. 
Hence, $\mmm(\mmW)$ is $ \bigcup_{B \in \mmQ(\mmW)} q(B)$-\rel, that is, $\mmQ(\mmW)$-\rel.

(b) By Proposition~\ref{pro:monomial}(d),
\[
\del_S \, \mmm(\mmW) =  \bigcup_{B \in \mmQ(\mmW)} \del_S \, \mmm(B)
=
\bigcup_{B \in \mmQ(\mmW)} \mmq(B)(S^*) = \mmQ(\mmW)(S^*).
\]
\end{proof}

 By Propositions~\ref{pro:maincoset} and~\ref{pro:w_monomial}(b), we obtain:
\[
\mmm(\mmW) \text{ is \injS }
\quad \Leftrightarrow \quad
0 \notin \mmQ(\mmW) (S^*) .
\]
Hence,
for the rectangular domain $X = \R^n_>$, 
injectivity of $\mmW$-monotonic maps and 
injectivity of $\mmW$-monomial maps (with exponent matrix in the qualitative class of $\mmW$) are equivalent.

Motivated by applications,
we additionally consider a matrix $A \in \R^{m \times r}$ and the resulting sets of composed maps.
By Proposition~\ref{pro:framecomp},
the equation $\del_S \, \mmM(\mmW) = \del_S \, \mmm(\mmW) = \mmQ(\mmW)(S^*)$ shown above, equivalence~\eqref{eq:ker_AB_S}, and Equation~\eqref{eq:qualitative_class_union}, 
we obtain: 
\begin{proposition} 
Let  $\mmW\in \mmS^{r \times n}$,
$\mmM(\mmW) \subseteq \map{\R^n_>}{\R^r}$ be the set of $\mmW$-monotonic maps,
and $\mmm(\mmW) \subseteq \map{\R^n_>}{\R^r}$ be the set of $\mmW$-monomial maps.
Further, let $S \subseteq \R^n$ be a linear subspace
and $A \in \R^{m \times r}$.
The following statements are equivalent:
\begin{enumerate}[(i)]
\item $A \circ \mmM(\mmW)$ is \injS.
\item $A \circ \mmM(W)$ is \injS for all $W\in \{-,0,+\}^{r\times n}$ with $W\subseteq \mmW$.
\item $A \circ m(\mmW)$ is \injS.
\item 
$\ker(A) \cap \mmQ(\mmW)(S^*) = \emptyset$. 
\end{enumerate}
\end{proposition}

 Cf.~\cite[Theorem~10.1]{WiufFeliu_powerlaw}, where $S=\im(A)$ and (iv) is expressed by a determinant condition (``$\mmW$ is $A$-sign-non-singular'').

\subsubsection{Monotonic maps and chemical reaction networks} \label{sec:chem}

In the study of chemical reaction networks~\cite{ShinarFeinberg2012,WiufFeliu_powerlaw},
one considers dynamical systems $\dd{x}{t} = A \, G(x)$, given a matrix $A \in \R^{n \times r}$ (a ``stoichiometric matrix''), a matrix of sign combinations $\mmW \in \mmS^{r \times n}$ (based on an ``influence specification''),  and a $\mmW$-monotonic map $G \in \mmM(\mmW) \subseteq (\R^n_\ge \to \R^r_\ge)$ (a ``weakly monotonic kinetics''). If the set of maps $A\circ \mmM(\mmW)$ is \inj on cosets of the linear subspace $S=\im(A)$, then the corresponding dynamical systems have at most one equilibrium in every coset.

The following result determines whether $A\circ \mmM(\mmW)$ is \inj on cosets of a linear subspace $S$,
that is, by Proposition~\ref{pro:framecomp} and equivalence~\eqref{eq:ker_AB_S}, whether
\[
\ker(A) \cap \mmQ(\mmW)(S^*) = \emptyset ,
\]
or negatively, whether there are $B \in \mmQ(\mmW)$, $y \in \ker(A)$, and $x \in S^*$ such that $y = B x$,
cf.~statement (i) in Proposition~\ref{pro:concordant} below.

\begin{notation}[Sign vectors continued]
The product on $\{-,0,+\}$ is defined in the obvious way. 
For $\tau, \rho \in \{-,0,+\}^n$,
we write $\tau \cdot \rho = 0$ ($\tau$ and $\rho$ are orthogonal) if
either $\tau_i \rho_i = 0$ for all $i$
or there exist $i,j$ with $\tau_i \rho_i = -$ and $\tau_j \rho_j = +$.
Equivalently,
$\tau \cdot \rho = 0$
if there exist $u,v \in \R^n$ with $\sigma(u)=\tau$, $\sigma(v)=\rho$, and $u \cdot v = 0$.
Moreover, as it is easy to see,
if $\tau \cdot \sigma(v) = 0$ for $\tau \in \{-,0,+\}^n$ and $v \in \R^n$, then there exists $u \in \R^n$ with $\sigma(u) = \tau$ and $u \cdot v = 0$.
For $w \in \mmS^n$ and $\rho \in \{-,0,+\}^n$,
we write $w \cdot \rho = 0$ if
there is $\tau \in w$ with $\tau \cdot \rho = 0$.
\end{notation}

\begin{proposition}\label{pro:concordant}
Let $\mmW=(w_{ij}) \in \mmS^{r \times n}$, $x \in \R^n$, and $y \in \R^r$.
The following statements are equivalent:
\begin{enumerate}[(i)]
\item There exists $B\in \mmQ(\mmW)$ such that $y=Bx$.
\item
For all $i \in [r]$, \\
if $y_i \neq 0$, then $\sigma(y)_i = s \, \sigma(x)_j$ with $s \in w_{ij}$ for some $j \in [n]$; \\
if $y_i = 0$, then $w_{i*} \cdot \sigma(x) = 0$, where $w_{i*}$ denotes the $i$-th row of  $\mmW$.
\end{enumerate}
\end{proposition}

\begin{proof}
Assume (i),
that is, there exists $B \in \mmQ(\mmW)$ such that $y_i = \sum_j b_{ij} x_j$ for all $i \in [r]$.
If $y_i \neq 0$, then $\sigma(y)_i = \sign(b_{ij} x_j) = s \, \sigma(x)_j$ with $s \in w_{ij}$ for some $j \in [n]$.
If $y_i = 0$, then $\sigma(b_{i*}) \cdot \sigma(x) = 0$. Now, $\sigma(b_{i*}) \in w_{i*}$ and hence $w_{i*} \cdot \sigma(x) = 0$.

Conversely, assume (ii)
and construct $B \in \mmQ(\mmW)$ such that $y = B x$.
If $y_i \neq 0$, then $\sigma(y)_i = s \, \sigma(x)_j$ with $s \in w_{ij}$ for some $j \in [n]$.
Let $\epsilon>0$ and, for all $j' \neq j$, set $b_{ij'} = -\epsilon,$ 0, or $+\epsilon$ such that $\sign(b_{ij'}) \in w_{ij'}$.
Choose $\epsilon$ small enough such that $|\sum_{j' \neq j} b_{ij'} x_{j'}| < |y_i|$
and determine $b_{ij}$ from $y_i - \sum_{j' \neq j} b_{ij'} x_{j'} = b_{ij} x_j$, where $\sign(b_{ij}) = s \in w_{ij}$.
If $y_i = 0$, then $w_{i*} \cdot \sigma(x) = 0$
and there exists $\tau \in w_{i*}$ with $\tau \cdot \sigma(x) = 0$.
By the argument above, there exists $u \in \R^n$ with $\sigma(u) = \tau$ and $u \cdot x = 0$.
Now, set $b_{i*}= u$ and $\sigma(b_{i*}) = \tau \in w_{i*}$, by construction.
\end{proof}

In the study of chemical reaction networks, a criterion for injectivity of $A \circ \mmM(\mmW)$ on cosets of $S=\im(A)$ was obtained, for particular $\mmW\in \mmS^{r\times n}$ \cite[Proposition 9.18]{ShinarFeinberg2012}.
The condition involves statement~(ii) in Proposition~\ref{pro:concordant} above.
Networks meeting the criterion are called ``concordant'' (with respect to the influence specification). The relation between concordance and $\mmW$-monotonicity was studied in \cite[Section~12]{WiufFeliu_powerlaw}.


\section{Maps with partial derivatives, differentiable maps, and the use of mean value theorems} \label{sec:diff}

We revisit results that guarantee injectivity of individual maps (having partial derivatives or being differentiable), by using the univariate mean-value theorem or a corresponding result for vector-valued functions.

Let $X \subseteq \R^n$ have non-empty interior and $G \colon X \to \R^r$ have partial derivatives.
We define $J_G \subseteq \R^{r \times n}$, the set of Jacobian matrices, as
\[
J_G = \left\{ \DD{G}{x} (x) \mid x \in X \right\}
\]
and write $\conv(J_G)$ for its convex hull.
 Further, we define $D_G$, an $r \times n$ matrix of non-empty subsets of $\R$, component-wise as 
\[
(D_G)_{ij}  = \left\{ \DD{G_i}{x_j}(x) \mid x \in X \right\} .
\]

Using the univariate mean value theorem,
we observe that maps $G$ that are continuous and have partial derivatives are component-wise $D_G$-\rel.
See also \cite[Theorem~5]{LagrangeDelanoueJaulin2007}.

\begin{proposition} \label{pro:partial}
Let $X \subseteq \R^n$ be convex with non-empty interior,
and let $G \colon X \to \R^r$ be continuous and have partial derivatives. 
Then $G$ is component-wise $D_G$-\rel.
\end{proposition}
\begin{proof}
Let $i\in [r]$,  $j\in [n]$, and $x,y\in X$ that differ only in the $j$-th component.
Let $X_j\subseteq \R$ denote the projection  of $X$ onto the  $j$-th component. Since $X$ is convex, $X_j$ is also convex,
and $[x_j,y_j] \subseteq X_j$.
Now, consider the univariate map
\begin{align*}
g_{ij} \colon & [x_j,y_j] \to \R, \\
& z \mapsto G_i(x_1,\dots,x_{j-1},z,x_{j+1},\dots,x_n).
\end{align*}
The mean value theorem yields
\[
G_i(x) - G_i(y) = g_{ij} (x_j) - g_{ij}(y_j) = \dd{g_{ij}}{x_j}(\xi)(x_j-y_j) = \DD{G_i}{x_j}(z) (x_j-y_j) 
\]
for some $\xi \in (x_j,y_j)$ and where $z =(x_1,\dots,x_{j-1},\xi,x_{j+1},\dots,x_n) \in X$.
Hence, $G$ is component-wise $D_G$-\rel. 
\end{proof}
Under the hypotheses of Proposition~\ref{pro:partial},
the map $G$ is component-wise $D_G$-\rel.
By Theorem~\ref{thm:D} (for a suitable domain)
and equivalence~\eqref{eq:ker_B_S},
$G$ is $\mmQ(D_G)$-\rel
and \inj if $\ker(\mmQ(D_G)) = \{0\}$.

Using a mean value theorem for vector-valued functions stated in~\cite{McLeod1965},
we find that differentiable maps $G$ are $\conv(J_G)$-\rel.

\begin{theorem}[cf.~Theorem 4 in \cite{McLeod1965}] \label{thm:mean_value}
Let $F \colon [a,b] \to \R^r$ be continuous and differentiable on $(a,b)$.
Then,
\[
F(b)-F(a) = \sum_{i=1}^{r} \lambda_i \, F'(\xi_i) (b-a) 
\]
with $\xi_i \in (a,b)$, $\lambda_i \ge 0$, and $\sum_{i=1}^{r} \lambda_i = 1$.
\end{theorem}

\begin{proposition} \label{pro:diff}
Let $X \subseteq \R^n$ be open and convex, and let $G \colon X \to \R^r$ be differentiable.
Further, let $x,y \in X$. Then,
\[
G(x)-G(y) = B (x-y)
\]
for some $B \in \conv(J_G)$.
In particular, $B$ is the convex combination of at most $r$ Jacobian matrices on the line segment between $x$ and $y$.
Hence, $G$ is $\conv(J_G)$-\rel.
\end{proposition}

\begin{proof}
Let $\varphi \colon [0,1] \to X$, $t \mapsto tx+(1-t)y$ and $F=G \circ \varphi\colon [0,1] \to \R^r$.
Then, by the chain rule,
\[
F'(t) = \DD{G}{x} (\varphi(t)) (x-y) .
\]
By Theorem~\ref{thm:mean_value},
\[
G(x)-G(y) = F(a)-F(0) = \sum_{i=1}^{r} \lambda_i \, F'(\xi_i) = \sum_{i=1}^{r} \lambda_i \, \DD{G}{x} (z_i) (x-y)
\]
with $\xi_i \in (0,1)$, $z_i = \xi_i x+(1-\xi_i)y$ lying on the line segment between $x$ and $y$,
$\lambda_i \ge 0$, and $\sum_{i=1}^{r} \lambda_i = 1$.
\end{proof}

Under the hypotheses of Proposition~\ref{pro:diff}, 
the map $G$ is $\conv(J_G)$-\rel.
By Proposition~\ref{pro:main} and equivalence~\eqref{eq:ker_B_S},
$G$ is injective
if $\ker(\conv(J_G)) = \{0\}$. 
This is essentially the statement of Corollary 2.1 in \cite{Coomes1989}, which follows from the most general injectivity result in \cite{McLeod1965}, Theorem~9.

\subsection{Discussion}

Let $X \subseteq \R^n$ be open and convex (and hence connected by rectangular, oriented paths),
and let $G \colon X \to \R^r$ have continuous partial derivatives. Then the entries of $D_G$ are connected, that is, intervals of $\R$. 
As shown above, 
$G$ is both $\mmQ(D_G)$-\rel and $\conv(J_G)$-\rel,
and both sets of matrices can be used to determine injectivity of $G$ (on cosets of a linear subspace $S$).
Note that 
\begin{equation*} 
J_G \subseteq \conv(J_G) \subseteq \mmQ(D_G) ,
\end{equation*}
where the latter inclusion follows from $J_G \subseteq \mmQ(D_G)$ and the convexity of $\mmQ(D_G)$,
and where both inclusions can be strict. 
By Proposition~\ref{pro:maincoset},
\[
0 \notin \mmQ(D_G)(S^*)
\quad \Rightarrow \quad
0 \notin \conv(J_G)(S^*)
\quad \Rightarrow \quad
G \text{ is \injS } .
\]
In order to guarantee injectivity, it is sufficient (but difficult) to determine the convex hull of the set of Jacobian matrices $J_G$.
It is easier to determine the interval matrix of partial derivatives, $D_G$.
Moreover, by Theorem~\ref{thm:D},
the whole class $\mmM(D_G)$, the set of all component-wise $D_G$-\rel maps on~$X$, is \injS
{\em if and only if} $0 \notin \mmQ(D_G)(S^*)$.

\begin{remark}
Let $X\subseteq \R^n$ be an open and convex set, and let $G \colon X \to \R^n$ be a polynomial map of degree at most two. 
Then each entry of the Jacobian matrix $\DD{G}{x}(x)$ is a polynomial of degree at most one. 
It follows that the set $J_G$ is convex and hence $G$ is $J_G$-\rel.
Proposition~\ref{pro:maincoset} implies that $G$ is \inj 
if $\det(\DD{G}{x}(x)) \neq 0$ for all $x \in X$,
that is, if the Jacobian matrix is non-singular.

This problem is related to the \emph{real Jacobian conjecture} for polynomial maps of degree at most two;
cf.~\cite[Theorem~62]{Wang:1980cc} and \cite[Theorem~2.4]{Bass:1982p770}.
\end{remark}

\begin{example} \label{ex}
Let $X=\R^2_>$ and $G \colon X \to \R^2$ 
be given by
\[
G(x_1,x_2)=(x_1 x_2,x_1^2 x_2).
\]
Then 
\[
J_G=\left\{
\begin{pmatrix} 
x_2 & x_1 \\ 2x_1 x_2 & x_1^2
\end{pmatrix}
\mid x \in \R^2_>
\right \} ,\quad 
D_G=
\begin{pmatrix} 
(0,+\infty) & (0,+\infty) \\ (0,+\infty) & (0,+\infty)
\end{pmatrix} ,
\]
and
\[
\mmQ(D_G) = \left\{
\begin{pmatrix} 
\k_1 & \k_2 \\ \k_3 & \k_4
\end{pmatrix}
\mid  
\k_1,\k_2,\k_3,\k_4>0  
\right \}.
\]
The map $G$ is both $\conv(J_G)$-\rel
and $\mmQ(D_G)$-\rel.
In this case, we have strict inclusions
\[
J_G \subset \conv(J_G) \subset \mmQ(D_G).
\]
To see this, note that the projection of $\conv(J_G)\subseteq \R^{2\times 2}$ onto the second column is the convex hull of $C:=\{(x_1,x_1^2)^T \mid x_1 \in \R_>\}$. Since $C$ is not convex, we conclude $J_G \neq \conv(J_G)$, and since $\conv(C) \subset \R^2_>$, we conclude $\conv(J_G) \neq \mmQ(D_G)$.

Let $S \subseteq \R^2$ be a linear subspace.
By Theorem~\ref{thm:D} and equivalence~\eqref{eq:ker_B_S}, 
the class $\mmM(D_G)$ is \injS if and only if $\ker(\mmQ(D_G)) \cap S = \{0\}$. 
Since $\mmQ(D_G)$ contains singular matrices,
$\mmM(D_G)$ is not \inj (for $S=\R^2$). However, $\ker(\mmQ(D_G))\cap S=\{0\}$ for $S= \im(1,1)^T$,
and $\mmM(D_G)$ is \injS.

Note that the map $G$ is in fact the generalized monomial map $\mu_B$ with
\[
B = \begin{pmatrix}
1 & 1 \\
2 & 1
\end{pmatrix}
\]
which is $\mmq(B)$-\rel with
\[
q(B) = 
\left\{ 
\begin{pmatrix} 
\k_1\lambda_1 & \k_1\lambda_2 \\ 
\k_2\lambda_1 & 2\k_2\lambda_2 
\end{pmatrix} 
\mid \k_1,\k_2,\lambda_1,\lambda_2\in \R_> \right\}
\]
and $\mmq(B) \subset \mmQ(D_G)$.
Now, $\mu_B$ is \inj if and only if $\ker(\mmq(B)) = \{0\}$,
cf.~Subsection~\ref{sec:monomial}.
Since all matrices in $\mmq(B)$ are non-singular, $\mu_B$ is \inj.

By using $q(B)$, having dependent entries, instead of $\mmQ(D_G)$, having independent entries,
we have concluded that $G$ is \inj.
\end{example}

\begin{example} \label{ex:jac} 
Let $X=\R^2_>$ and $G\colon X \to \R^2$ be given by
\[
G(x_1,x_2)=(x_1,x_1^2+x_2^2).
\]
Then 
\[
J_G = \left\{
\begin{pmatrix} 
1 & 0 \\ 2 x_1 & 2 x_2
\end{pmatrix}
\mid
x \in \R^2_>
\right\},
\]
and we obtain
\begin{align*}
\mmQ(D_G) & =  \left\{ 
 \begin{pmatrix} 
1 & 0 \\ \k_1 & \k_2
\end{pmatrix} 
\mid  
\k_1,\k_2>0  
\right \}  =  J_G.
\end{align*}
Hence, $J_G = \conv(J_G) = \mmQ(D_G)$.
Now, $\ker(\mmQ(D_G)) = \{0\}$ and hence the class $\mmM(D_G)$ is \inj.

Note that the map $G$ is $W$-monotonic with
\[
W= \begin{pmatrix} 
+ & 0 \\ + & +
\end{pmatrix} \in \{-, 0, +\}^{2 \times 2} 
\]
and $\mmQ(W)$-\rel with 
\[
\mmQ(W)=
\left\{  
\begin{pmatrix} 
\k_1 & 0 \\  \k_2 & \k_3
\end{pmatrix}  
\mid \k_1,\k_2,\k_3>0   
\right \} 
\]
and $\mmQ(D_G) \subset \mmQ(\mmW)$.  
Still, $\ker(\mmQ(D_G)) = \ker(\mmQ(W))=\{0\}$, 
and hence $\mmM(D_G)$ and the even larger class $\mmM(W)$ are injective.
\end{example}

\begin{remark} \label{rem:mon}
Generalized monomial maps on the positive orthant are differentiable.
The Jacobian matrix of $\k \had x^B \in \mmm(B)$ evaluated at $x$ is given by
\[
\diag(\k) \diag(x^B) B \diag(x^{-1}) \in \mmq(B).
\]
Conversely, the matrix $\diag(\k') B \diag(\lambda) \in \mmq(B)$
agrees with the Jacobian matrix of $\k' \had \lambda^B \had x^B \in \mmm(B)$ 
evaluated at $x=\lambda^{-1}$.
Hence, 
\[
J_{\mmm(B)} = \{ J_G \mid G \in \mmm(B) \} = \mmq(B) \subseteq \conv(J_{\mmm(B)}).
\] 
In general, 
the inclusion is strict.  
For the matrix $B$ in Example~\ref{ex}, the set $\mmq(B)$
agrees with the set of positive matrices $A=(a_{ij}) \in \R_>^{2\times 2}$ with $a_{11}a_{22} - 2a_{12}a_{21}=0$ which is not convex.
As a consequence, the results in Subsection~\ref{sec:monomial} do not follow from this section: $\ker(\conv(J_{\mmm(B)})) = \{0\}$ just implies injectivity of $\mmm(B)$, whereas $\ker(\mmq(B)) = \{0\}$ characterizes injectivity.
\end{remark}

The use of $\conv(J_G)$ or 
$\mmQ(D_G)$ 
allows to derive {\em domain-dependent} injectivity criteria.

\begin{example}[Example~\ref{ex} continued]
Let $X=(\ell_1,u_1) \times (\ell_2,u_2)\subseteq \R^2_>$ with $\ell_1<u_1$ and $\ell_2<u_2$.
Then
\[
\mmQ(D_G) = \left\{
\begin{pmatrix} 
\k_1 & \k_2 \\ \k_3 & \k_4
\end{pmatrix}
\mid
\begin{array}{l}
\k_1\in (\ell_2,u_2), \, \k_2\in (\ell_1,u_1), \\ 
\k_3\in (2\ell_1\ell_2,2u_1u_2), \, \k_4\in (\ell_1^2,u_1^2)\end{array} \right\}.
\]
The determinant of any matrix in $\mmQ(D_G)$ is 
negative if $u_1^2 u_2 < 2 \ell_1^2 \ell_2$.
In this case, $0\notin \mmQ(D_G)$ and hence $G$ is \inj on $X$.
\end{example}

\begin{remark} 
For continuously differentiable maps, 
the essence of Proposition~\ref{pro:diff} can be obtained by invoking the fundamental theorem of calculus instead of the mean value theorem for vector-valued maps. This approach has been used, for example, in \cite{gouze,BanajiPantea2016}. Let $X \subseteq \R^n$ be open and convex, and let $G \colon X \to \R^r$ be continuously differentiable. 
For $x,y\in X$ and $t \in [0,1]$, let $F(t)= G(tx+(1-t)y)$. By the fundamental theorem of calculus,
\[
G(x)-G(y)=\left( \int_{0}^1 F'(t) dt \right)(x-y).
\]
Hence, for any set of matrices $\mmB$ containing the integrals $\int_{0}^1 F'(t) dt $ for all pairs $x,y\in X$, $G$ is $\mmB$-\rel.  In particular, $\conv(J_G)$ has this property \cite[Lemma 3.11]{BanajiPantea2016}.
\end{remark}

\begin{remark}
Let $X\subseteq \R^n$ be a closed rectangular domain, and let $G\colon X \to \R^n$ be differentiable.
A result of Gale and Nikaid\^o \cite[Theorem 4]{Gale:1965p474} states that if $J_G$ consists of $P$-matrices, that is, of matrices having all principal minors positive, then $G$ is \inj on $X$. This result does not follow from our framework.
In particular, a matrix on the line segment between two P-matrices can be singular.
\end{remark}


\section{Examples of composed maps} \label{sec:compose}

Compositions of generalized monomial and monotonic maps with a matrix have been studied already in Subsections~\ref{sec:polynom} and \ref{sec:monoton_monom}, \ref{sec:chem}.
In the following, we consider more examples of composed maps.
In particular, we study the injectivity of rational functions.

\subsection{Composition with a matrix}

First, we study the composition of a map $G$ (from $X \subseteq \R^n$ to $\R^r$) with a matrix $A$ (in $\R^{n \times r}$).   
By using Proposition~\ref{pro:framecomp}, 
sets of maps suitable to guarantee injectivity of $G$ can be used to guarantee injectivity of $A \circ G$.

\begin{example} \label{ex:G1}
Let $X=\R^2_>$, and let $H \colon X\to \R^2$ be given by
\[
H(x_1,x_2)= \left(\frac{x_2}{1+x_2}-x_1,\frac{x_1}{1+x_1}-x_2\right). 
\]
Then $H=A \circ G$ with 
\[
G(x_1,x_2) = 
\left(x_1,\frac{x_1}{1+x_1},x_2,\frac{x_2}{1+x_2}\right) ,
\quad 
A =
\begin{pmatrix} 
-1 & 0 & 0 & 1 \\ 0 & 1 & -1 & 0 
\end{pmatrix}
\]
and 
\[
D_G=
\begin{pmatrix} 
1 & 0 \\ (0,1) & 0 \\ 0 & 1 \\ 0 & (0,1) 
\end{pmatrix}, 
\quad 
\mmQ(D_G) =  
\left\{ 
\begin{pmatrix} 
1 & 0 \\ \k_1 & 0 \\ 0 & 1 \\ 0 & \k_2
\end{pmatrix} 
\mid  
\k_1,\k_2\in (0,1)
\right\} .
\]
Hence, $G \in \mmM(D_G)$ and $H \in A \circ \mmM(D_G)$.
Let $S \subseteq \R^2$
and recall $dX \cap S = S^*$ for $X=\R^2_>$.
By Theorem~\ref{thm:D}, $\del \mmM(D_G) = \mmQ(D_G)(S^*)$;
and by Proposition~\ref{pro:framecomp} and equivalence~\eqref{eq:ker_AB_S},
$A \circ \mmM(D_G)$ is \inj
if and only if
$\ker(A \mmQ(D_G)) \cap S = \{0\}$.
Now, any matrix in
\begin{align*}
A\mmQ(D_G) & = 
\left \{
\begin{pmatrix} 
-1 & \k_1 \\ \k_2 & -1
\end{pmatrix}  
\mid \k_1,\k_2\in (0,1) \right \}
\end{align*}
has positive determinant, 
$1 -\k_1 \k_2 > 0$.
Hence, $A \circ \mmM(D_G)$ is \inj.
\end{example}

\begin{example}[Variant of Example~\ref{ex:G1}] \label{ex:G1continued}
Let $X=\R^2_>$, and let $H \colon X\to \R$ be given by
\[
H(x_1,x_2) = \frac{x_1}{1+x_1} - x_2 .
\]
Then $H = A \circ G$ with
\[
G(x_1,x_2)=
\left( \frac{x_1}{1+x_1}, x_2 \right), \quad
A = 
\begin{pmatrix}
1 & -1
\end{pmatrix}
\]
and
\[
D_G =
\begin{pmatrix} 
(0,1) & 0 \\ 0 & 1 
\end{pmatrix} , \quad
\mmQ(D_G)= 
\left\{  
\begin{pmatrix} 
\k_1 & 0 \\ 0 & 1 
\end{pmatrix} 
\mid \k_1 \in (0,1) 
\right\} .
\]
As in Example~\ref{ex:G1},
$A \circ \mmM(D_G)$ is \injS
if and only if
$\ker(A \mmQ(D_G)) \cap S = \{0\}$.
Since $A \mmQ(D_G)= \begin{pmatrix} \k_1 & -1 \end{pmatrix}$ has non-trivial kernel, $A \circ \mmM(D_G)$ is not \inj (for $S=\R^2$). 
However, for the linear subspace $S = \im (1,1)^T$, $A \circ \mmM(D_G)$ is \injS.
\end{example}

\begin{example}[Example~\ref{ex:jac} continued]
The map $G \colon \R^2_> \to \R^2$, $G(x_1,x_2) = (x_1,x_1^2+x_2^2)$ can be written as
\[
G(x) = A\, x^B
\quad \text{with} \quad
A = \begin{pmatrix} 1 & 0 & 0 \\ 0 &1 & 1 \end{pmatrix} ,
\quad
B = \begin{pmatrix} 
1 & 0 \\ 2 & 0 \\ 0 & 2 \end{pmatrix}.
\]
Then
\[
q(B)= \left\{
\begin{pmatrix} 
\k_1\lambda_1 & 0 \\ 
2 \k_2 \lambda_1 & 0 \\ 
0 & 2 \k_3\lambda_2
\end{pmatrix}
\mid \k\in \R^3_>,\lambda\in \R^2_> 
\right \} 
\]
and
\[
A \, q(B) = 
\left\{  
\begin{pmatrix} 
\k_1\lambda_1 & 0 \\ 
2 \k_2 \lambda_1 & 2 \k_3\lambda_2
\end{pmatrix}
\mid \k\in \R^3_>,\lambda\in \R^2_> 
\right \}  
= \mmQ(W), 
\]
see Example~\ref{ex:jac}.
Hence, to guarantee injectivity of $G$, it is equivalent to consider 
$G$ as a map in $\mmM(W)$ or $A\circ m(B)$:

By Theorem~\ref{thm:D} with $S=\R^2$ and equivalence~\eqref{eq:ker_B_S},
$\mmM(W)$ is injective if and only if $\ker(\mmQ(W)) = \{0\}$.
By \eqref{eq:mon_poly} with $S=\R^2$ and equivalence~\eqref{eq:ker_AB_S},
$A \circ \mmm(B)$ is \inj
if and only if
$\ker(A \, \mmq(B)) = \{0\}$.
\end{example}

\subsection{Rational functions as compositions}

We consider rational functions and sets of maps suitable to characterize/guarantee their injectivity.
More generally, we consider compositions of monomial maps, namely functions $H \colon \R^n_> \to \R^m$ 
of the forms
\[
H(x) = 
F(\k \had x^B) \; \in \; \mmF \circ \mmm(B)
\]
with $B \in \R^{r \times n}$, $\k \in \R^r_>$, $F \in \mmF \subseteq (\R^r_> \to \R^m)$
and
\[
H(x) = 
\k \had G(x)^{A} \; \in \; \mmm(A) \circ \mmG
\]
with $G \in \mmG \subseteq (\R^n_> \to \R^r)$, $A \in \R^{m \times r}$, $\k \in \R^m_>$.

\begin{proposition} \label{prop:compose:monomials}
Let $S\subseteq \R^n$ be a linear subspace.
\begin{enumerate}[(a)]
\item
Let $B\in \R^{r\times n}$, $\mmF \subseteq \map{\R^r_>}{\R^m}$, and $\mmA \subseteq \R^{m\times r}$.
If $\mmF$ is $\mmA$-\rel and $\mmA \subseteq \mmF$, then
\[
\del_S (\mmF \circ m(B)) = \mmA q(B)(S^*).
\]
\item 
Let $\mmG \subseteq \map{\R^n_>}{\R^r_>}$, $\mmB \subseteq \R^{r \times n}$, and $A \in \R^{m\times r}$. If $\del_S \, \mmG \subseteq \mmB (S^*)$,
then
\[
\del_S (\mmm(A)\circ \mmG) \subseteq \mmq(A) \mmB (S^*).
\]
\end{enumerate}
\end{proposition}
\begin{proof}
(a) By Lemma~\ref{lem:comp}(b) and Proposition~\ref{pro:monomial}(d). \\
(b) By Lemma~\ref{lem:comp}(a) and Proposition~\ref{pro:monomial}(c). 
\end{proof}

\begin{example} \label{ex:comp1}
Let $H \colon \R^3_> \to \R^2$ be given by
\[
H(x_1,x_2,x_3)= 
\left( 
\frac{x_1x_2}{1+x_3},x_1x_2x_3^2  
\right).
\]
Clearly, $H$ is a function of the monomials $x_1x_2$ and $x_3$. In particular, $H(x)=F(x^B)$
with
\[
B=  
\begin{pmatrix} 
1 & 1 & 0 \\  0 & 0 & 1  
\end{pmatrix} ,\quad 
F(y_1,y_2)=  
\left( \frac{y_1}{1+y_2},y_1y_2^{2}  
\right).
\]
The map $F$ is $W$-monotonic with
\begin{align*}
W &= 
\begin{pmatrix} 
+ & - \\ + & + 
\end{pmatrix} ,
\end{align*}
and hence $H \in \mmM(W) \circ m(B)$.
By Theorem~\ref{thm:D},
$\mmM(W)$ is $\mmQ(W)$-\rel
and $\mmQ(W) \subseteq \mmM(W)$.
Let $S \subseteq \R^3$ be a linear subspace.
By Propositions~\ref{pro:maincoset} and \ref{prop:compose:monomials}(a),
$\mmM(W) \circ m(B)$ is \injS if and only if
$0 \notin \mmQ(W) \mmq(B) (S^*)$,
that is, $\ker( \mmQ(W)\mmq(B)) \cap S = \{0\}$, by equivalence~\eqref{eq:ker_B_S}.
Now,%
\begin{gather*}
\mmq(B) =  
\left\{  
\begin{pmatrix} 
\k_1\lambda_1 & \k_1\lambda_2 & 0 \\ 0 & 0 & \k_2\lambda_3
\end{pmatrix}
\mid \k_1,\k_2,\lambda_1,\lambda_2,\lambda_3>0 
\right\}  \\ 
= 
\left\{  
\begin{pmatrix} 
\lambda_1 & \lambda_2 & 0 \\ 0 & 0 &\lambda_3
\end{pmatrix}
\mid \lambda_1,\lambda_2,\lambda_3>0 
\right\}, \\
\mmQ(W) =  
\left\{  
\begin{pmatrix} 
\mu_1 & -\mu_2 \\ \mu_3 & \mu_4
\end{pmatrix}
\mid  \mu_1,\mu_2,\mu_3,\mu_4>0 \right\} ,
\intertext{and hence}
\mmQ(W) \mmq(B) = 
\left\{  
\begin{pmatrix} 
\mu_1\lambda_1 & \mu_1\lambda_2 & -\mu_2\lambda_3 \\
\mu_3 \lambda_1 &  \mu_3 \lambda_2 & \mu_4 \lambda_3
\end{pmatrix}
\mid \mu_1,\mu_2,\mu_3>0,\lambda_1,\lambda_2,\lambda_3>0
\right\}.  
\end{gather*}
Let $S\subseteq \R^3$ be given by $x_1-x_2+x_3=0$. 
The determinant of the matrix 
\[
\begin{pmatrix} 
1 & -1 & 1 \\ 
\mu_1 \lambda_1 & \mu_1 \lambda_2 & -\mu_2 \lambda_3 \\
\mu_3 \lambda_1 & \mu_3 \lambda_2 &  \mu_4 \lambda_3
\end{pmatrix}
\] 
is strictly positive for all parameters,
and hence $\mmM(W) \circ m(B)$
is \injS,
see Equation~\eqref{total_determinental}.

Finally, $H$ is also a map in $\mmM(W_H)$ with
\[
W_H = \begin{pmatrix}
+ & + & - \\ + & + & +
\end{pmatrix}.
\]
However, $0 \in \mmQ(W_H)(S^*)$ and hence $\mmM(W_H)$ is not injective on cosets of $S$, see~\eqref{eq:MW}.
\end{example}


\paragraph{Acknowledgments.}
EF was supported by a Sapere Aude Starting Grant from the Danish Research Council for Independent Research.
SM was supported by the Austrian Science Fund (FWF), project 28406. GR was supported by the FWF, project 27229.


\small

\end{document}